\documentclass[11pt,reqno]{amsart}
\usepackage{tikz}
\usepackage{subfig}
\usepackage{epstopdf}
\usepackage{epsfig}
\usepackage{math}
\graphicspath{{./figures/}}
\usepackage{tikz}
\usetikzlibrary{shapes,arrows}
\tikzstyle{decision} = [diamond, draw, fill=blue!20, 
    text width=4.5em, text badly centered, node distance=3cm, inner sep=0pt]
\tikzstyle{block} = [rectangle, draw, fill=blue!20, 
    text width=5em, text centered, rounded corners, minimum height=4em]
\tikzstyle{line} = [draw, -latex']
\tikzstyle{cloud} = [draw, ellipse,fill=red!20, node distance=3cm,
    minimum height=2em]

\newcommand{\rt}{\mathrm{T}}
\usetikzlibrary{positioning}
\tikzset{main node/.style={circle,fill=blue!20,draw,minimum size=1cm,inner sep=0pt},  }

\begin{document}
\title{Transport information Hessian distances}
\author[Li]{Wuchen Li}
\email{wuchen@mailbox.sc.edu}
\address{Department of Mathematics, University of South Carolina, Columbia, 29208.}

\begin{abstract}
We formulate closed-form Hessian distances of information entropies in one-dimensional probability density space embedded with the $L^2$-Wasserstein metric. 
\end{abstract}
\maketitle
\section{Introduction}
Hessian distances of information entropies in probability density space play crucial roles in information theory with applications in signal and image processing, inverse problems, and AI \cite{IG, CY, CoverThomas1991_elements}. 
One important example is the Hellinger distance, known as a Hessian distance of negative Boltzmann-Shannon entropy in $L^2$ (Euclidean) space. Nowadays, the Hellinger distance has shown various useful properties in AI inference problems. 

Recently, optimal transport distance, a.k.a. Wasserstein distance\footnote{There are various generalizations of optimal transport distances using different ground costs defined on a sample space. For the simplicity of discussion, we focus on the $L^2$-Wasserstein distance.}, provides the other type of distance functions in probability density space \cite{AGS, Villani2009_optimal}. Unlike the $L^2$ distance, the Wasserstein distance compares probability densities by their pushforward mapping functions. More importantly, it introduces a metric space under which the information entropies present convexity properties in term of mapping functions. These properties nowadays have been shown useful in fluid dynamics, inverse problems and AI.

Nowadays, it reveals that the optimal transport distance itself is a Hessian distance of a second moment functional in Wasserstein space. Natural questions arise. {\em Can we construct Hessian distances of information entropies in Wasserstein space? 
And what is the Hessian distance of negative Boltzmann-Shannon entropy in Wasserstein space?}

In this paper, we derive closed-form Hessian distances of information entropies in Wasserstein space supported on a one dimensional sample space. In details, given a compact set $\Omega\subset\mathbb{R}^1$, consider a (negative) $f$-entropy by
\begin{equation*}
\mathcal{F}(p)=\int_\Omega f(p(x))dx,
\end{equation*}
where $f\colon \mathbb{R}\rightarrow\mathbb{R}$ is a second differentiable convex function and $p$ is a given probability density function. We show that the Hessian distance of $f$-entropy in Wasserstein space between probability density functions $p$ and $q$ satisfies 
\begin{equation*}
\mathrm{Dist}_{\mathrm{H}}(p,q)=\sqrt{\int_0^1\|h(\nabla_yF_p^{-1}(y))- h(\nabla_yF_q^{-1}(y))\|^2dy},
\end{equation*}
where $h(y)=\int_1^y\sqrt{f''(\frac{1}{z})}\frac{1}{z^{\frac{3}{2}}}dz$, $F_p$, $F_q$ are cumulative density functions (CDFs) of $p$, $q$ respectively, and $F^{-1}_p$, $F^{-1}_q$ are their inverse CDFs.  
We call $\mathrm{Dist}_{\mathrm{H}}(p,q)$ {\em transport information Hessian distances}. Shortly, we show that the proposed distances are constructed by the Jacobi operators of mapping functions between density functions $p$ and $q$. 

This paper is organized as follows. In section \ref{sec2}, we briefly review the Wasserstein space and its Hessian operators for information entropies. In section \ref{sec3}, we derive closed-form solutions of Hessian distances in Wasserstein space. Several analytical examples are presented in section \ref{sec4}. 
\section{Review of Transport information Hessian metric}\label{sec2}
In this section, we briefly review the Wasserstein space and its induced Hessian metrics for information entropies. 
\subsection{Wasserstein space}
We recall the definition of a one dimensional Wasserstein distance \cite{AGS}. 
Denote a spatial domain by $\Omega=[0,1]\subset\mathbb{R}^1$. Consider the space of smooth probability densities by 
\begin{equation*}
\mathcal{P}(\Omega)=\Big\{p(x)\in C^{\infty}(\Omega)\colon \int_\Omega p(x)dx=1,\quad p(x)\geq 0\Big\}.
\end{equation*}
Given any two probability densities $p$, $q\in\mathcal{P}(\Omega)$, the squared Wasserstein distance in $\Omega$ is defined by 
\begin{equation*}
\begin{split}
\textrm{Dist}_{\rt}(p,q)^2=&\int_\Omega\|T(x)-x\|^2q(x)dx,
\end{split}
\end{equation*}
where $\|\cdot\|$ represents a Euclidean norm and $T$ is a monotone mapping function such that $T(x)_\#q(x)=p(x)$, i.e. 
\begin{equation*}
p(T(x))\nabla_xT(x)=q(x).
\end{equation*}
Since $\Omega\subset \mathbb{R}^1$, then the mapping function $T$ can be solved analytically. Concretely, 
\begin{equation*}
T(x)=F_p^{-1}(F_q(x)),
\end{equation*}
where $F^{-1}_p$, $F^{-1}_q$ are inverse CDFs of probability densities $p$, $q$, respectively. Equivalently,
\begin{equation*}
\begin{split}
\textrm{Dist}_{\rt}(p,q)^2=&\int_\Omega \|F^{-1}_p(F_q(x))-x\|^2q(x)dx\\
=&\int_0^1\|F^{-1}_p(y)-F^{-1}_q(y)\|^2dy,
\end{split}
\end{equation*}
 where we apply the change of variable $y=F_q(x)\in [0,1]$ in the second equality. 
 
 There is a metric formulation for the $L^2$-Wasserstein distance. Denote the tangent space of $\mathcal{P}(\Omega)$ at a probability density $p$ by 
\begin{equation*}
T_p\mathcal{P}(\Omega)=\Big\{\sigma \in C^{\infty}(\Omega)\colon \int_\Omega \sigma(x) dx=0 \Big\}.
\end{equation*}
And the $L^2$-Wasserstein metric refers to the following bilinear form:
\begin{equation*}
g_{\mathrm{T}}(p)(\sigma, \sigma)=\int_\Omega (\nabla_x\Phi(x), \nabla_x\Phi(x))p(x)dx,
\end{equation*}
where $\Phi, \sigma \in C^{\infty}(\Omega)$ satisfy an elliptical equation 
\begin{equation}\label{ee}
\sigma(x)=-\nabla_x\cdot(p(x)\nabla_x\Phi(x)),
\end{equation}
with either Neumann or periodic boundary conditions on $\Omega$. Here, the above mentioned boundary conditions ensure the fact that $\sigma$ stays in the tangent space of probability density space, i.e. $\int_\Omega \sigma(x)dx=0$. As a known fact, the Wasserstein metric can be derived from a Taylor expansion of the Wasserstein distance. I.e.,
\begin{equation*}
\mathrm{Dist}_\mathrm{T}(p, p+\sigma)^2=g_{\mathrm{T}}(\sigma, \sigma)+o(\|\sigma\|_{L^2}^2), 
\end{equation*}
where $\sigma\in T_p\mathcal{P}(\Omega)$ and $\|\cdot\|_{L_2}$ represents the $L^2$ norm.  In literature, $(\mathcal{P}(\Omega), g_\rt)$ is often called the {\em Wasserstein space}.   
\subsection{Hessian metrics in Wasserstein space}
We next review the Hessian operator of $f$-entropies in Wasserstein space \cite{Villani2009_optimal}; see a derivation in appendix. Consider a (negative) $f$-entropy by 
\begin{equation*}
 \mathcal{F}(p)=\int_\Omega f(p(x))dx,
 \end{equation*} 
 where $f\colon \mathbb{R}\rightarrow\mathbb{R}$ is a one dimensional second order differentiable convex function. The Hessian operator of $f$-entropy in Wasserstein space is a bilinear form satisfying 
\begin{equation*}
\mathrm{Hess}_{\mathrm{T}}\mathcal{F}(p)(\sigma, \sigma)=\int_\Omega \|\nabla^2_x\Phi(x)\|^2f''(p(x)) p(x)^2dx.
\end{equation*}
where $f''$ represents the second derivative of function $f$ and $(\Phi, \sigma)$ satisfies an elliptic equation \eqref{ee}. 

In this paper, we seek closed-form solutions for transport Hessian distances of $\mathcal{F}$. It is to solve an action functional in $(\mathcal{P}(\Omega), \mathrm{Hess}_{\mathrm{T}}\mathcal{F})$. 
\begin{definition}[Transport information Hessian distance \cite{LiHess}]
 Define a distance function $\mathrm{Dist}_{\mathrm{H}}$ $\colon \mathcal{P}(\Omega)\times\mathcal{P}(\Omega)\rightarrow\mathbb{R}$ by
\begin{equation}\label{HC}
\mathrm{Dist}_{\mathrm{H}}(p,q)^2=\inf_{p\colon [0,1]\times\Omega\rightarrow\mathbb{R}}\Big\{\int_0^1\mathrm{Hess}_{\mathrm{T}}\mathcal{F}(p)(\partial_t p, \partial_t p)dt\colon p(0,x)=q(x),~p(1,x)=p(x)\Big\}.
\end{equation}
Here the infimum is taken among all smooth density paths $p\colon [0, 1]\times \Omega\rightarrow\mathbb{R}$, which connects both initial and terminal time probability density functions $q$, $p\in \mathcal{P}(\Omega)$. 
\end{definition}
From now on, we call $\mathrm{Dist}_{\mathrm{H}}$ the {\em transport information Hessian distance}. Here the terminology of ``transport'' corresponds to the application of Wasserstein space, while the name of ``information'' refers to the usage of $f$-entropies. 

\section{Transport information Hessian distances}\label{sec3}
In this section, we present the main result of this paper. 
\subsection{Formulations}
We first derive closed-form solutions for transport information Hessian distances defined by \eqref{HC}. 
\begin{theorem}
Denote a one dimensional function $h\colon \Omega\rightarrow\mathbb{R}$ by
\begin{equation*}
h(y)=\int_1^y\sqrt{f''(\frac{1}{z})}\frac{1}{z^{\frac{3}{2}}}dz.
\end{equation*}
Then the squared transport Hessian distance of $f$-entropy has the following formulations. 
\begin{itemize}
\item[(i)] Inverse CDF formulation: 
\begin{equation*}
\begin{split}
\mathrm{Dist}_{\mathrm{H}}(p, q)^2=&\int_0^1 \|h({\nabla_yF_p^{-1}(y)})-h({\nabla_yF_q^{-1}(y)})\|^2dy.
\end{split}
\end{equation*}
\item[(ii)] Mapping formulation: \begin{equation*}
\begin{split}
\mathrm{Dist}_{\mathrm{H}}(p, q)^2=&\int_\Omega \|h(\frac{\nabla_xT(x)}{q(x)})-h(\frac{1}{q(x)})\|^2q(x)dx,
\end{split}
\end{equation*}
where $T$ is a mapping function, such that $T_\#q =p$ and $T(x)=F_p^{-1}(F_q(x))$. Equivalently,  
\begin{equation*}
\begin{split}
\mathrm{Dist}_{\mathrm{H}}(p, q)^2=&\int_\Omega \|h(\frac{\nabla_xT^{-1}(x)}{p(x)})-h(\frac{1}{p(x)})\|^2p(x)dx,
\end{split}
\end{equation*}
where $T^{-1}$ is the inverse function of mapping function $T$, such that $(T^{-1})_\#p =q$ and $T^{-1}(x)=F_q^{-1}(F_p(x))$.
\end{itemize}
\end{theorem}
\begin{proof}
We first derive formulation (i). To do so, we consider the following two set of change of variables. 

Firstly, denote $p^0(x)=q(x)$ and $p^1(x)=p(x)$. Denote the variational problem \eqref{HC} by 
\begin{equation*}
\begin{split}
\mathrm{Dist}_{\mathrm{H}}(p^0,p^1)^2=\inf_{\Phi,p\colon[0,1]\times \Omega\rightarrow\mathbb{R}}&\Big\{\int_0^1\int_\Omega\| \nabla^2_y\Phi(t,y)\|^2f''(p(t,y))p(t,y)^2dydt\colon \\
&\hspace{1cm}\partial_tp(t,y)+\nabla\cdot(p(t,y) \nabla_y\Phi(t,y))=0,~\textrm{fixed $p^0$, $p^1$}\Big\},
\end{split}
\end{equation*}
where the infimum is among all smooth density paths $p\colon [0, 1]\times\Omega\rightarrow\mathbb{R}$ satisfying the continuity equation with gradient drift vector fields $\nabla\Phi\colon[0,1]\times \Omega\rightarrow\mathbb{R}$. 
Denote 
\begin{equation*}
y=T(t,x),
\end{equation*}
and let  
\begin{equation*}
\partial_tT(t,x)=\nabla_y\Phi(t,y).
\end{equation*}
Hence 
\begin{equation*}
\begin{split}
\mathrm{Dist}_{\mathrm{H}}(p^0,p^1)^2=&\inf_{T\colon[0,1]\times\Omega\rightarrow \Omega}\Big\{\int_0^1\int_\Omega\|\nabla_y v(t,T(t,x))\|^2f''(p(t, T(t,x)))p(t,T(t,x))^2dT(t,x)dt\colon\\
&\hspace{3cm} T(t,\cdot)_\#p(0,x)=p(t,x)\Big\},
\end{split}
\end{equation*}
where the infimum is taken among all smooth mapping functions $T\colon [0,1]\times \Omega \rightarrow \Omega$ with $T(0,x)=x$ and $T(1,x)=T(x)$. 
We observe that the above variation problem leads to 
\begin{equation}\label{variation1}
\begin{split}
&\int_0^1\int_\Omega\|\nabla_y \partial_tT(t,x)\|^2f''(p(t, T(t,x)))p(t,T(t,x))^2\nabla_xT(t,x)dxdt\\
=&\int_0^1\int_\Omega\|\nabla_x \partial_tT(t,x)\frac{dx}{dy}\|^2f''(p(t, T(t,x)))p(t,T(t,x))^2\nabla_xT(t,x)dxdt\\
=&\int_0^1\int_\Omega\|\nabla_x \partial_tT(t,x)\frac{1}{\nabla_xT(t,x)}\|^2f''(\frac{p(0, x)}{\nabla_xT(t,x)})\frac{p(0,x)^2}{\nabla_xT(t,x)}dxdt\\
=&\int_0^1\int_\Omega\|\partial_t\nabla_x T(t,x)\frac{1}{(\nabla_xT(t,x))^{3/2}}\sqrt{f''(\frac{q(x)}{\nabla_xT(t,x)})}\|^2q(x)^2dxdt.
\end{split}
\end{equation}

Secondly, denote $y=F_q(x)$, where $y\in[0,1]$. By using a chain rule for $T(t,\cdot)_\#q(x)=p_t$ with $p_t:=p(t,x)$, we have
\begin{equation*}
q(x)=\frac{dy}{dx}=\frac{1}{\frac{dx}{dy}}=\frac{1}{\nabla_yF_q^{-1}(y)},
\end{equation*}
and 
\begin{equation*}
\nabla_xT(t,x)=\nabla_xF_{p_t}^{-1}(F_q(x))=\nabla_yF_{p_t}^{-1}(y)\frac{dy}{dx}=\frac{\nabla_yF_{p_t}^{-1}(y)}{\nabla_yF_q^{-1}(y)}.
\end{equation*}
Under the above change of variables, variation problem \eqref{variation1} leads to 
\begin{equation}\label{variation}
\begin{split}
\mathrm{Dist}_{\mathrm{H}}(p^0,p^1)^2=&\inf_{\nabla_yF_{p_t}^{-1}\colon [0,1]^2\rightarrow\mathbb{R}}\Big\{\int_0^1\int_0^1\|\partial_t\nabla_yF_{p_t}^{-1}(y)\frac{1}{(\nabla_yF_{p_t}^{-1}(y))^{\frac{3}{2}}}\sqrt{f''(\frac{1}{\nabla_yF_{p_t}^{-1}(y)})}\|^2dydt\Big\}\\
=&\inf_{\nabla_yF_{p_t}^{-1}\colon [0,1]^2\rightarrow\mathbb{R}}\Big\{\int_0^1\int_0^1 \|\partial_t h(\nabla_yF_{p_t}^{-1}(y))\|^2dydt\Big\},
\end{split}
\end{equation}
where the infimum is taken among all paths $\nabla_yF_{p_t}^{-1}\colon [0,1]^2\rightarrow\mathbb{R}$ with fixed initial and terminal time conditions. Here we apply the fact that 
\begin{equation*}
\nabla_yh(y)=\sqrt{f''(\frac{1}{y})}\frac{1}{y^{\frac{3}{2}}},
\end{equation*}
and treat $\nabla_yF_{p_t}^{-1}$ as an individual variable. By using the Euler-Lagrange equation for $\nabla_yF_{p_t}^{-1}$, we show that the geodesic equation in transport Hessian metric satisfies 
\begin{equation*}
\partial_{tt}h(\nabla_yF_{p_t}^{-1}(y))=0. 
\end{equation*}
In details, we have
\begin{equation*}
h(\nabla_yF_{p_t}^{-1}(y))=th(\nabla_yF_{p}^{-1}(y))+(1-t)h(\nabla_yF_{q}^{-1}(y)),
\end{equation*}
and 
\begin{equation*}
\partial_th(\nabla_yF_{p_t}^{-1}(y))=h(\nabla_yF_{p}^{-1}(y))-h(\nabla_yF_{q}^{-1}(y)). 
\end{equation*}
Therefore, variational problem \eqref{variation} leads to the formulation (i). 

We next derive formulation (ii). Denote $y=F_q(x)$ and $T(x)=F_p^{-1}(F_q(x))$. By using formulation (i) and the change of variable formula in integration, we have 
\begin{equation*}
\begin{split}
\mathrm{Dist}_{\mathrm{H}}(p^0,p^1)^2=&\int_0^1 \|h({\nabla_yF_p^{-1}(y)})-h({\nabla_yF_q^{-1}(y)})\|^2dy\\
=&\int_\Omega\|h(\nabla_yF_p^{-1}(F_q(x))- h(\nabla_y x) \|^2dF_q(x)\\
=&\int_\Omega \|h(\frac{\nabla_xT(x)}{\frac{dy}{dx}})-h(\frac{1}{\frac{dy}{dx}})\|^2q(x)dx\\
=&\int_\Omega \|h(\frac{\nabla_xT(x)}{q(x)})-h(\frac{1}{q(x)})\|^2q(x)dx. 
\end{split}
\end{equation*}
This finishes the first part of proof. Similarly, we can derive the transport information Hessian distance in term of the inverse mapping function $T^{-1}$. 
\end{proof}
\begin{remark}
We notice that $\mathrm{Dist}_{\mathrm{H}}$ forms a class of distance functions in probability density space. Compared to the classical optimal transport distance, it emphasizes on the differences by the Jacobi operators of mapping operators. In this sense, we call the transport information Hessian distance the {\em ``Optimal Jacobi transport distance''}. 
\end{remark}
\begin{remark}
 Transport information Hessian distances share similarities with transport Bregman divergences defined in \cite{LiB}. Here we remark that transport Hessian distances are symmetric to $p$, $q$, i.e. $\mathrm{Dist}_{\mathrm{H}}(p,q)=\mathrm{Dist}_{\mathrm{H}}(q, p)$, while the 
transport Bregman divergences are often asymmetric to $p$, $q$; see examples in \cite{LiB}.  
\end{remark}

\subsection{Properties}
We next demonstrate that transport Hessian distances have several basic properties. 
\begin{proposition}
The transport Hessian distance has the following properties. 
\begin{itemize}
\item[(i)] Nonnegativity: 
\begin{equation*}
\mathrm{Dist}_{\mathrm{H}}(p,q)\geq 0.
\end{equation*}
In addition,
\begin{equation*}
\mathrm{Dist}_{\mathrm{H}}(p,q)=0\quad \textrm{iff}\quad p(x+c)=q(x), 
\end{equation*}
 where $c\in\mathbb{R}$ is a constant. 
\item[(ii)] Symmetry: 
\begin{equation*}
\mathrm{Dist}_{\mathrm{H}}(p,q)=\mathrm{Dist}_{\mathrm{H}}(q,p).
\end{equation*}
\item[(iii)] Triangle inequality: For any probability densities $p$, $q$, $r\in\mathcal{P}(\Omega)$, we have 
\begin{equation*}
\mathrm{Dist}_{\mathrm{H}}(p,r)\leq \mathrm{Dist}_{\mathrm{H}}(p,q)+\mathrm{Dist}_{\mathrm{H}}(q,r).
\end{equation*}
\item[(v)] Hessian metric: Consider a Taylor expansion by 
\begin{equation*}
\mathrm{Dist}_{\mathrm{H}}(p,p+\sigma)^2=\mathrm{Hess}_{\mathrm{T}}\mathcal{F}(p)(\sigma, \sigma)+o(\|\sigma\|_{L^2}^2),
\end{equation*}
where $\sigma\in T_p\mathcal{P}(\Omega)$.
\end{itemize}
\end{proposition}
\begin{proof}
From the construction of transport Hessian distance, the above properties are satisfied. Here we only need to show (i). $\mathrm{Dist}_{\mathrm{H}}(p,q)=0$ implies that 
$\|h(\nabla_xT(x))-h(1)\|=\|h(\nabla_xT(x))\|=0$ under the support of density $q$.  Notice that $h$ is a monotone function. Thus $\nabla_xT(x)=1$. Hence $T(x)=x+c$, for some constant $c\in\mathbb{R}$. From the fact that $T_\#q =p$, we derive $p(T(x))\nabla_xT(x)=q(x)$. This implies $p(x+c)=q(x)$. 
\end{proof}
\section{Closed-form distances}\label{sec4}
In this section, we provide several closed-form examples of transport information Hessian distances.  
From now on, we always denote 
\begin{equation*}
y=F_q(x)\quad\textrm{and}\quad T(x)=F_p^{-1}(F_q(x)).
\end{equation*}
\begin{example}[Boltzmann-Shanon entropy]
Let $f(p)=p\log p$, i.e.
\begin{equation*}
\mathcal{F}(p)=-\mathcal{H}(p)=\int_\Omega p(x)\log p(x) dx.
\end{equation*}
Then $h(y)=\log y$. Hence
\begin{equation}\label{TH}
\begin{split}
\mathrm{Dist}_{\mathrm{H}}(p,q)=&\sqrt{\int_\Omega\|\log\nabla_xT(x)\|^2q(x)dx}\\
=&\sqrt{\int_0^1\|\log\nabla_y F^{-1}_p(y)-\log\nabla_y F^{-1}_q(y)\|^2dy}.
\end{split}
\end{equation}
\end{example}
\begin{remark}[Comparisons with Hellinger distances]
We compare the Hessian distance of Boltzmann-Shannon entropy defined in either $L^2$ space or Wasserstein space. In $L^2$ space, this Hessian distance is known as the Hellinger distance, where
$\mathrm{Hellinger}(p,q)=\sqrt{\int_\Omega \|\sqrt{p(x)}-\sqrt{q(x)}\|^2dx}$. Here distance \eqref{TH} is an analog of the Hellinger distance in Wasserstein space. 
\end{remark}
\begin{example}[Quadratic entropy]
Let $f(p)=\frac{1}{2}p^2$, then $h(y)=-2(y^{-\frac{1}{2}}-1)$. Hence
\begin{equation*}
\begin{split}
\mathrm{Dist}_{\mathrm{H}}(p,q)=&2\sqrt{\int_\Omega\| (\nabla_x T(x))^{-\frac{1}{2}}-1 \|^2q(x)^2dx}\\
=&2\sqrt{\int_0^1\|(\nabla_y F_p^{-1}(y))^{-\frac{1}{2}}-(\nabla_y F^{-1}_q(y))^{-\frac{1}{2}}\|^2dy}.
\end{split}
\end{equation*}
\end{example}

\begin{example}[Cross entropy]
Let $f(p)=-\log p$, then $h(y)=2(y^{\frac{1}{2}}-1)$. Hence 
\begin{equation*}
\begin{split}
\mathrm{Dist}_{\mathrm{H}}(p,q)=&2\sqrt{\int_\Omega\|(\nabla_xT(x))^{\frac{1}{2}}-1\|^2dx}\\
=&2\sqrt{\int_0^1\|(\nabla_yF^{-1}_p(y))^{\frac{1}{2}}-(\nabla_yF_q^{-1}(y))^{\frac{1}{2}}\|^2dy}.
\end{split}
\end{equation*}

\end{example}

\begin{example}
Let $f(p)=\frac{1}{2p}$, then $h(y)=y-1$. Hence 
\begin{equation*}
\begin{split}
\mathrm{Dist}_{\mathrm{H}}(p,q)=&\sqrt{\int_\Omega\|\nabla_xT(x)-1\|^2q(x)^{-1}dx}\\
=&\sqrt{\int_0^1\|\nabla_yF^{-1}_p(y)-\nabla_yF_q^{-1}(y)\|^2dy}.
\end{split}
\end{equation*}

\end{example}

\begin{example}[$\gamma$-entropy]
Let $f(p)=\frac{1}{(1-\gamma)(2-\gamma)}p^{2-\gamma}$, $\gamma\neq 1$, $2$, then $h(y)=\frac{2}{\gamma-1}(y^{\frac{\gamma-1}{2}}-1)$. Hence 
\begin{equation*}
\begin{split}
\mathrm{Dist}_{\mathrm{H}}(p,q)=&\frac{2}{|\gamma-1|}\sqrt{\int_\Omega\|(\nabla_x T(x))^{\frac{\gamma-1}{2}}-1\|^2q(x)^{2-\gamma}dx}\\
=&\frac{2}{|\gamma-1|}\sqrt{\int_0^1\|(\nabla_yF_p^{-1}(y))^{\frac{\gamma-1}{2}}-(\nabla_yF^{-1}_q(y))^{\frac{\gamma-1}{2}}\|^2dy}.
\end{split}
\end{equation*}
\end{example}

\noindent{\textbf{Acknowledgement}: W. Li is supported by a start-up funding in University of South Carolina.}

\section*{Appendix}
For the completeness of this paper, we present the derivation of Hessian operator of $f$-entropy in Wasserstein space below. 
\begin{proposition}[Formula 15.7 in \cite{Villani2009_optimal}] Denote $\mathcal{F}(p)=\int_\Omega f(p(x))dx$ and $\Omega\subset\mathbb{R}^1$. Then 
\begin{equation*}
\mathrm{Hess}_{\mathrm{T}}\mathcal{F}(p)(\sigma, \sigma)=\int_\Omega \|\nabla_x^2\Phi(x)\|^2 f''(p(x))p(x)^2dx,
\end{equation*}
where $\sigma(x)=-\nabla\cdot(p(x)\nabla_x\Phi(x))$.
\end{proposition}
\begin{proof}
Notice that the geodesics in Wasserstein space satisfies  
\begin{equation*}
\left\{\begin{split}
&\partial_t p(t,x)+\nabla_x\cdot(p(t,x)\nabla_x \Phi(t,x))=0\\
&\partial_t \Phi(t,x)+\frac{1}{2}\|\nabla_x\Phi(t,x)\|^2=0,
\end{split}\right.
\end{equation*}
where $p(0,x)=p(x)$ and $\partial_tp(0,x)=\sigma(x)=-\nabla_x\cdot(p(x)\nabla_x\Phi(x))$.  
We only need to show that $\mathrm{Hess}_{\mathrm{T}}\mathcal{F}(p)(\sigma, \sigma)=\frac{d^2}{dt^2}\mathcal{F}(p(t,\cdot))|_{t=0}$. The proof follows by a direct calculation. Denote $k(p)=f'(p)p-f(p)$. Then the first order derivative of $\mathcal{F}$ w.r.t. $t$ forms 
\begin{equation*}
\begin{split}
\frac{d}{dt}\mathcal{F}(p(t,\cdot))=&-\int_\Omega \nabla_x\cdot(p(t,x)\nabla_x\Phi(t,x))f'(p(t,x))dx\\
=&\int_\Omega \nabla_x\Phi(t,x)\nabla_xf'(p(t,x))p(t,x)dx\\
=&\int_\Omega \nabla_x\Phi(t,x) \nabla_x k(p(t,x))dx\\
=&-\int_\Omega \Delta_x\Phi(t,x) k(p(t,x))dx,
\end{split}
\end{equation*}
where the first and last equality use the integration by parts and the second equality apply the fact $k'(p)=f''(p)p$. And the second order derivative of $\mathcal{F}$ w.r.t. $t$ satisfies 
\begin{equation*}
\begin{split}
&\frac{d^2}{dt^2}\mathcal{F}(p(t,\cdot))|_{t=0}\\
=&-\int_\Omega \Big\{\Delta_x\partial_t\Phi(t,x) k(p(t,x))+\Delta_x\Phi(t,x) k'(p(t,x))\partial_t p(t,x)\Big\}dx|_{t=0}\\
=&\int_\Omega \Big\{\frac{1}{2}\Delta_x\|\nabla_x\Phi(x)\|^2 k(p(x))+\Delta_x\Phi(x)k'(p(x))\nabla_x\cdot(p(x)\nabla_x\Phi(x))\Big\}dx\\
=&\int_\Omega \Big\{\frac{1}{2}\Delta_x\|\nabla_x\Phi\|^2 k(p(x))+\Delta_x\Phi(x) k'(p(x))(\nabla_xp(x), \nabla_x\Phi(x))+(\Delta\Phi(x))^2k'(p(x))p(x)\Big\}dx\\
=&\int_\Omega \Big\{\frac{1}{2}\Delta_x\|\nabla_x\Phi\|^2 k(p(x))+\Delta_x\Phi(x)(\nabla_xk(p(x)), \nabla_x\Phi(x))+(\Delta\Phi(x))^2k'(p(x))p(x)\Big\}dx.
\end{split}
\end{equation*}
We notice that 
\begin{equation*}
\begin{split}
\int_\Omega \Delta_x\Phi(x)(\nabla_xk(p(x)), \nabla_x\Phi(x)) dx=&-\int_\Omega \nabla_x\cdot(\nabla_x\Phi(x)\Delta_x\Phi(x)) k(p(x))dx\\
=&-\int_\Omega \Big[(\nabla_x\Delta_x\Phi(x), \nabla_x\Phi(x))+(\Delta_x\Phi(x))^2\Big]k(p(x))dx.
\end{split}
\end{equation*}
Combining the above two formulas and using a Bochner's formula  
\begin{equation*}
\frac{1}{2}\Delta_x\|\nabla_x\Phi(x)\|^2-\nabla_x\Delta_x\Phi(x)\nabla_x\Phi(x)=\|\nabla_x^2\Phi(x)\|^2,
\end{equation*}
we finish the proof. 
\end{proof}
\end{document}